\documentclass[12pt]{article}
\usepackage{amsmath,amsfonts,amsthm,amssymb}
\usepackage{hyperref}

\newcommand{\stsets}[1]{\mathbb{#1}}
\newcommand{\R}{\stsets{R}}

\theoremstyle{definition}
\newtheorem{definition}{Definition}
\theoremstyle{remark}

\newtheoremstyle{mytheorem}{0.5cm}{0.2cm}{\slshape}{ }{\bfseries}{.}{ }{}
\theoremstyle{mytheorem}
\newtheorem{Th}[definition]{Theorem}

\newtheorem{Lem}[definition]{Lemma}

\renewcommand{\P}{\mathbf{P}}

\DeclareMathOperator{\E}{{\bf E}}
\DeclareMathOperator{\var}{{\bf var}}
\DeclareMathOperator{\cov}{{\bf var}}

\DeclareMathOperator{\one}{{ 1\hspace*{-0.55ex}I}}
\newcommand{\thru}{,\dotsc,}
\newcommand{\bydef}{\stackrel{{\rm def}}{=}}

\renewcommand{\epsilon}{\varepsilon}

\renewcommand{\phi}{\varphi}
\newcommand{\ti}{\to\infty}
\newcommand{\ssp}{\hspace{0.5pt}}
\newcommand{\seg}{see, \hbox{e.\ssp g.,}}

\newcommand{\ou}{\overline{u}}
\newcommand{\om}{\overline{m}}

\newlength{\querylen}
\usepackage{fancybox}

\begin{document}

\title{
Functional central limit theorems for occupancies and missing mass
process in infinite urn models}
\author{Mikhail Chebunin\thanks{
    Sobolev Institute of Mathematics SB RAS and Novosibirsk State
    University, Novosibirsk, Russia, E-mail: chebuninmikhail@gmail.com} \and
Sergei Zuyev \thanks{Chalmers University of Technology, 
Gothenburg, Sweden. E-mail: sergei.zuyev@chalmers.se}}
\date{\today}
\maketitle

\begin{abstract}
We study the infinite urn scheme when the balls are sequentially
distributed over an infinite number of urns labelled 1,2,\dots so that
the urn $j$ at every draw gets a ball with probability $p_j$, $\sum_j p_j=1$. 
We prove functional central limit theorems for discrete time and the
poissonised version for the urn occupancies
process, for the odd-occupancy and for the missing mass processes
extending the known non-functional central limit theorems.


\end{abstract}

\section{Introduction}
In this paper we study the following classical urn model first
considered by Karlin~\cite{Karlin}: $n\geq 1$ balls are distributed
one by one over an infinite number of urns enumerated from 1 to
infinity. The ball distributed at step $j=1,2\dots$, call it $j$th
ball, gets into urn $i$ with probability $p_i$,
$\sum_{i=1}^\infty p_i=1$, independently of the other balls. Such
multinomial occupancy schemes arise in many different applications,
in Biology \cite{GooTou:56}, Computer science \cite{M-Z}, \cite{OSZ}
and in many other areas, \seg~\cite{Gnedin} and the references
therein.

Let $X_j$ be the urn the $j$th ball gets into and let $J_i(n)$
be the number of balls the $i$th urn contains after $n$ balls are
distributed:
\begin{displaymath}
  J_i(n)=\sum_{j=1}^{n} \one_{X_j=i}.
\end{displaymath}
Of a particular interest is the asymptotic behaviour of the following
quantities: the
number of urns containing at least $k\geq1$ balls and containing
exactly $k$ balls: 
\begin{equation}\label{def1}
R^{*}_{n,k}=\sum_{i=1}^{\infty} \one_{J_i(n)\ge k},\quad
R_{n,k}=\sum_{i=1}^{\infty} \one_{J_i(n)= k}=R^{*}_{n,k}-R^{*}_{n,k+1}, 
\end{equation}
the number of urns with an odd number of balls and the \emph{scaled
  missing mass} introduced in \cite{Karlin}:
\begin{equation}\label{def2}
U_n=\sum_{i=1}^\infty \one_{J_i(n)\equiv 1\ (\mathrm{mod} \ 2)},
\quad M_n=n \sum_{i=1}^\infty p_i \one_{J_i(n)=0}, 
\end{equation}
We also use notation $R_n\bydef R^{*}_{n,1}=\sum_{k\ge 1} R_{n,k}$ for
the number of non-empty urns. Renumbering the urns if necessary, we
may assume that the sequence $(p_i)_{i\geq1}$ is monotonely
decaying. We further assume that it is regularly varying:
\begin{equation}
\label{al}
\alpha(x)=\max\{i:\ p_i\geq 1/x\}=x^{\theta} L(x) \  \textrm{with
  $\theta\in [0,1]$,} 
\end{equation}
where $L(x)$ is a slowly varying function as $x\ti$.

Following Karlin's \cite{Karlin} original approach, we will consider a
Poissonised version of the model when the balls are put into urns at
the times of jumps of a homogeneous Poisson point processes
$\Pi(s), s\geq0$ with intensity 1 on $\R_+$.  According to the
independent marking theorem for Poisson processes,
$\{J_i(\Pi(s))\bydef\Pi_i(s), \ s\geq 0\}$ are independent homogeneous
Poisson processes with intensities $p_i$. To ease the notation, we
write simply
\begin{displaymath}
  R(s)\bydef R^{*}_{\Pi(s),1},\ U(s)\bydef U_{\Pi(s)},
\end{displaymath}
and we introduce the following poissonised version of the scaled missing mass:
\begin{displaymath}
  M(s)\bydef s \sum_{i=1}^\infty p_i \one_{\Pi_i(s)=0}. 
\end{displaymath}
It differs from $M_{\Pi(s)}$ by the scaling
factor $s$ vs.\ $\Pi(s)$, but, when properly scaled, it is
asymptotically equivalent to it.

Ordinary (not functional) central limit theorems for the above
quantities were established under various conditions in \cite{BG},
\cite{B-H}, \cite{Dutko}, \cite{Gnedin},% \cite{GooTou:56},
\cite{Karlin}, \cite{M-Z}, \cite{OSZ}. In particular, under rather
general conditions on the sequence $(p_i)$ involving an unbounded
growth of the variances, the following results are available: a strong
law of large numbers and asymptotic normality of $R_n$, an asymptotic
normality of the vector $(R_{n,1}\thru R_{n,\nu})$, local limit
theorems, etc.

We acknowledge a novel method of a randomised
decomposition for proving FCLTs developed in a recent paper
\cite{Durieu}, but we do not use it here. As a particular case of
their Theorem~2.3, a FCLT holds for the processes $R_n$ and $U_n$ when
$\theta\in(0,1)$. 

Our goal here is to establish a FCLT for the triplet of processes: the
occupancy, odd-occupancy and the scaled missing mass when
$\theta\in(0,1]$. In particular, we obtain previously unknown FCLT for
$U_n$ for $\theta=1$ and for $M_n$ when $\theta\in(0,1]$. Up to a
normalising constant, the FCLT stated in Theorem~\ref{Th1} also holds
for the original (non-scaled) missing mass
$\sum_{i=1}^\infty p_i \one_{J_i(n)=0}$ on any interval
$t\in[\epsilon,1]$, $\epsilon>0$, separated from~0.
The paper extends the results of \cite{Chebunin2} and
\cite{Chebunin3}, where a functional central limit theorem (FCLT) was
shown under condition \eqref{al} for the vector process\\
$(R^*_{[nt],1},R^*_{[nt],2},\dots, R^*_{[nt],\nu})_{t\in[0,1]}$ in
the case $\theta\in(0,1]$.

Extending the FCLT to the case $\theta=0$ would require additional to
\eqref{al} conditions. As it was mentioned in \cite{Karlin} and in
\cite{BG}, $\theta=0$ does not imply that the variances grow to
infinity and various asymptotic behaviour is possible for different
statistics. We also argue that even an infinite growth of variances
does not guarantee \emph{per se} the required relative compactness.

When $\theta=1$, we need a function
\begin{displaymath}
L^*(x)=\int_0^\infty L(xs) e^{-s}s^{-1} ds.
\end{displaymath}
It is known (see \cite{Karlin}) that $L^*(x)$ is slowly varying when
$x \to \infty$. 

Finally,  for $t\in [0,1]$ introduce the following notation:
\begin{align}\label{eq:betaR}
\beta(n)& =\left\{
\begin{array}{ll}
\alpha(n), & \theta\in[0,1); \\
nL^*(n), & \theta=1,
\end{array}
\right.
\quad
& R_n(t) =\frac{R_{[nt]}-\E R_{[nt]}}{(\beta(n))^{1/2}},\\
\label{eq:UM}
U_n(t) & =\frac{U_{[nt]}-\E U_{[nt]}}{(\beta(n))^{1/2}}, &
M_n(t) =\frac{M_{[nt]}-\E M_{[nt]}}{(\alpha(n))^{1/2}}.
\end{align}

We are now ready to formulate the main result of the paper.

\begin{Th} \label{Th1}
When $\theta\in(0, 1]$, the vector process
\begin{displaymath}
(R_n(t),U_n(t),M_n(t)),\quad t\in[0,1],
\end{displaymath}
converges weakly in the uniform metric on $D([0,1]^3)$ to a
3-dimensional Gaussian process $(\rho(t),\upsilon(t),\mu(t))$ with zero
mean and the covariance function $c(\tau,t)$ 
with the following components: when $\theta\in(0,1), \ \tau\le t$,
\begin{align*}
c_{\rho\rho}(\tau,t)& =\Gamma(1-\theta)((\tau+t)^\theta-t^\theta), \\
c_{\upsilon\upsilon}(\tau,t) & =
\Gamma(1-\theta)2^{\theta-2}((t+\tau)^\theta-(t-\tau)^\theta), 
\end{align*}
\begin{align*}
c_{\mu\mu}(\tau,t)&=\theta\Gamma(2-\theta) \left(
                    \frac{\tau}{t^{1-\theta}}-\frac{t\tau}{(t+\tau)^{2-\theta}}
                    \right), \\ 
c_{\rho\upsilon}(\tau,t)&=\Gamma(1-\theta)((2t+\tau)^{\theta} - (2t-\tau)^{\theta})/2, \\
c_{\rho\upsilon}(t,\tau)&=\Gamma(1-\theta)((2t+\tau)^\theta - t^\theta)/2,\\
c_{\rho\mu}(\tau,t)&=\theta \Gamma(1-\theta)\left( \frac{t}{(t+\tau)^{1-\theta}} - t^\theta \right), \\
c_{\rho\mu}(t,\tau)&=\theta \Gamma(1-\theta)\left( \frac{\tau}{(t+\tau)^{1-\theta}} -\frac{\tau}{ t^{1-\theta}} \right),\\
c_{\mu\upsilon}(\tau,t)&=\theta \Gamma(1-\theta)\left( \frac{\tau}{2(2t+\tau)^{1-\theta}} -\frac{\tau}{2(2t-\tau)^{1-\theta}} \right),\\
c_{\mu\upsilon}(t,\tau)&=\theta \Gamma(1-\theta)\left(
                \frac{t}{2(2\tau+t)^{1-\theta}} -\frac{t^\theta}{2}
                \right).
\end{align*}
When $\theta=1$, $ \tau\le t$, $c(\tau,t)$ is
given by
\begin{align*}
  c_{\rho\rho}(\tau,t)&= \tau, \ c_{\upsilon\upsilon}(\tau,t)=2\tau, \ c_{\mu\mu}(\tau,t)=\tau^2, \\
  c_{\rho\upsilon}(\tau,t)&=\tau, \ 
c_{\rho\upsilon}(t,\tau)=(t+\tau)/2,\\
c_{\rho\mu}(\tau,t)&=c_{\rho\mu}(t,\tau)=c_{\upsilon\mu}(\tau,t)=c_{\upsilon\mu}(t,\tau)=0.
\end{align*}
\end{Th}
Thus, when $\theta=1$, $\rho(t)$ and $\upsilon(t)$ are Wiener
processes. For a general $\theta\in(0,1]$, the process
$(\rho(t),\upsilon(t),\mu(t))$ is self-similar with the Hurst
parameter $H=\theta/2$ which includes, in particular, a fractional
Brownian motion, a bi-fractional Brownian motion with parameter
$H = 1/2, K = \theta$ (see, e.g.\ \cite{Durieu}) with a new
self-similar process $\mu(t)$.

\section{Proof of Theorem~\ref{Th1}}

We start with formulating a couple of lemmas proved
in~\cite{Chebunin2}. We will generally use the letter $C$ and its
variants to denote a constant whose value is of no importance for us
and note in parentheses the parameters it depends upon. This should
not lead to a confusion when the same notation is used for, actually,
different constants in different contexts, the same way $O(1)$
notation is used.
\begin{Lem}\label{Lem2.1}
When $\theta >0$, there exist $n_0\ge 1$ and $C(\theta)<\infty$ such that
\begin{displaymath}
\frac{\E R(n\delta)}{\beta(n)} 
\leq C(\theta) \delta^{\theta/2}
\end{displaymath}
holds for any $\delta\in[0,1]$ and $n\ge n_0$.
\end{Lem}

\begin{Lem}\label{Lem2.2}
For any $\epsilon, \delta\in(0,1)$ there exists an
$N=N(\epsilon,\delta)$ such that for any $n\ge N$,
\begin{displaymath}
  \P(\forall t\in[0,1] \  \ \exists \tau: |\tau-t|\le \delta,  
\ \Pi(n\tau)= [nt]) \ge1-\epsilon.
\end{displaymath}
\end{Lem}

In preparation of the proof, let us introduce some further
notation and establish a few inequalities we will be using.

In view of~\eqref{eq:UM}, let
\begin{align}
U_n^{*}(t) & =\frac{U (nt)-\E U (nt)}{(\beta(n))^{1/2}}, &
U_n^{**}(t)=\frac{U ([nt])-\E U ([nt])}{(\beta(n))^{1/2}}\\
M_n^{*}(t) & =\frac{M (nt)-\E M (nt)}{(\alpha(n))^{1/2}},  
& M_n^{**}(t)=\frac{M ([nt])-\E M ([nt])}{(\alpha(n))^{1/2}}.
\end{align}

For any two positive $\tau_1\le \tau_2$, define
\begin{align*}
U(\tau_2)-U(\tau_1)& =\sum_{i=1}^{\infty}\one \{\Pi_i(\tau_2)\ \text{is
  odd}\}-\one \{\Pi_i(\tau_1)\ \text{is odd}\}\\ 
& =\sum_{i=1}^{\infty}\one \{\Pi_i(\tau_2) \ \text{is odd}, \Pi_i(\tau_1) \ \text{is even}\}\\
& - \one \{\Pi_i(\tau_2) \ \text{is even}, \Pi_i(\tau_1) \ \text{is odd}\}\\
& \bydef \sum_{i=1}^{\infty}  u_{i}(\tau_1,\tau_2)
= \sum_{i=1}^{\infty}  u_{i}
  =\sum_{i=1}^{\infty} u'_{i}-u''_i,
\end{align*}
their expectations are denoted by
\begin{displaymath}
\ou_i =\ou_i'-\ou_i''=\ou_i(\tau_1,\tau_2)\bydef \E u'_i -\E u''_i.
\end{displaymath}
Similarly for $M$,
\begin{align*}
  M(\tau_2)-M(\tau_1)&=\sum_{i=1}^{\infty}(\tau_2-\tau_1)p_i\one
                       \{\Pi_i(\tau_2)=0\}- \tau_1 p_i\one
                       \{\Pi_i(\tau_1)=0,\Pi_i(\tau_2)>0\}\\ 
 & \bydef \sum_{i=1}^{\infty}
   m_{i}(\tau_1,\tau_2)=\sum_{i=1}^{\infty} m_{i}=\sum_{i=1}^{\infty}
   m'_{i}-m''_i,\\ 
 \om_i & =\om_i'-\om_i''=\om_i(\tau_1,\tau_2)\bydef \E m'_i -\E m''_i.
\end{align*}
Clearly, for all natural $k$,
\begin{align}
\E |u_i-\ou_i|^k &= |1+\ou_i|^k \ou_i''+|\ou_i|^k
                       (1-\ou_i'-\ou_i'')+|1-\ou_i|^k \ou_i' \notag \\
 & \le 2^k (\ou_i'+\ou_i'')+|\ou_i|^k\le (2^k+1) (\ou_i'+\ou_i'') \notag \\
& =(2^k+1)\Big[\sum_{j=0}^\infty \P \{\Pi_i(\tau_1)=2j, \
  \Pi_i(\tau_2)-\Pi_i(\tau_1)\ \text{is odd}\} \notag \\ 
& +\sum_{j=0}^\infty \P \{\Pi_i(\tau_1)=2j+1, \
  \Pi_i(\tau_2)-\Pi_i(\tau_1)\ \text{is odd}\}\Big] \notag \\
& =(2^k+1)\P \{\Pi_i(\tau_2-\tau_1)\ \text{is odd}\} \notag \\
 & <(2^k+1) \P
  \{\Pi_i(\tau_2-\tau_1)>0\}.  \label{u}
\end{align}
Similarly,
\begin{align*}
\E |m'_i-\om'_i|^k & \le 2^{k-1}(\E |m'_i|^k+|\om'_i|^k)=2^{k-1}(\tau_2-\tau_1)^kp^k_i(e^{-\tau_2 p_i}+e^{-k\tau_2 p_i})\\
& <2^k k! (1-e^{-(\tau_2-\tau_1)p_i})=  2^k k!\, \P \{\Pi_i(\tau_2-\tau_1)>0\},\\
\E |m''_i-\om''_i|^k & \le 2^{k-1}(\E |m''_i|^k+|\om''_i|^k)<2^{k}\tau_1^kp^k_i e^{-\tau_1 p_i}(1-e^{-(\tau_2-\tau_1)p_i})\\
& <2^k k! (1-e^{-(\tau_2-\tau_1)p_i})=  2^k k! \,\P \{\Pi_i(\tau_2-\tau_1)>0\}.
\end{align*}
As a result,
\begin{equation}\label{m}
\E |m_i-\om_i|^k<4^{k}k! \,\P \{\Pi_i(\tau_2-\tau_1)>0\}.
\end{equation}

We are using the same notation $u_i, \ m_i$ and $\ou_i, \ \om_i$
without explicitly specifying the corresponding values of
$\tau_1<\tau_2$, this should not create a confusion. The following
lemma will be used in the proof of the relative compactness of the
process $M^*_n(t)$.

\begin{Lem}\label{Lem2.3}
Let $\theta\in(0,1]$ and  $\delta\in[0,1]$. Then there exist $n_0\ge
1$ and $C(\theta)<\infty$ such that
\begin{displaymath}
\frac{\var (M(nt_2)-M(nt_1))}{\alpha(n)} \le C(\theta)\delta^{\theta/2}
\end{displaymath}
for all $t_2-t_1=\delta\ge 0$ and $n\ge n_0$.
\end{Lem}
\begin{proof}
Put $\tau_2=n t_2$ and $\tau_1=n t_1$. Since the variance of an
indicator does not exceed its expectation, we have that
\begin{multline*}
\var (M(\tau_2)-M(\tau_1))=\sum_{i=1}^{\infty}\E (m_i-\om_i)^2
= \sum_{i=1}^{\infty} \E (m'_i)^2-(\om'_i-\om_i'')^2+\E(m''_i)^2 \\
\le  \sum_{i=1}^{\infty} (\tau_2-\tau_1)^2p_i^2 e^{-\tau_2 p_i}
+\tau_1^2 p_i^2 e^{-\tau_1 p_i}(1-e^{-(\tau_2-\tau_1)p_i}\pm (\tau_2-\tau_1)p_i e^{-(\tau_2-\tau_1)p_i})\\
\le 2\frac{(\tau_2-\tau_1)^2}{\tau_2^2}\E R_{\Pi(\tau_2),2}+\E R^*_{\Pi(\tau_2-\tau_1),2}
+6\frac{\tau_1^2(\tau_2-\tau_1)}{\tau_2^3}\E R_{\Pi(\tau_2),3}.
\end{multline*}

By \cite[Th.~2.1\ and\ (23)]{Karlin},
\begin{displaymath}
 \lim\limits_{x\to\infty}\frac{\E R^*_{\Pi(x),2}}{\alpha(x)}=\Gamma(2-\theta)<2,
\end{displaymath}
therefore there exists an $x_1>1$ such that for all $x\ge x_1$, 
\begin{displaymath}
\E R_{\Pi(x),2}+\E R_{\Pi(x),3}<\E R^*_{\Pi(x),2}< 2 \alpha(x).
\end{displaymath}

According to Karamata (see, e.g.\ \cite[Th.~2,1,\
Eq.~A6.2.10]{BorovkovTV}), there exists an $x_2>0$ such that for all
$x$ and $\delta\in(0,1]$ satisfying $x\delta\ge x_2$, one has
\begin{displaymath}
  \frac{L(x\delta)}{L(x)}\le 2 \delta^{-1/2}.
\end{displaymath}
Let $n\delta>\max\{x_1,x_2\}=x_0$,  then
\begin{displaymath}
  \frac{\E R^*_{\Pi(n \delta),2}}{\alpha(n)} \le 
2\frac{(n \delta)^\theta L(n \delta)}{n^\theta L(n)}
\le
4 \delta^{\theta/2}, \ \ \ 
\frac{\max(\E R_{\Pi(n t_2),2},\E R_{\Pi(n t_2),3})}{\alpha(n)} 
\le 4 t_2^{\theta/2}.
\end{displaymath}
Choose $n_0$ such that for all $n\ge n_0$ we have $n^\theta L(n)\ge
n^{\theta/2}$. Then, provided $n t_2 \le x_0$,
\begin{align*}
& \frac{\E R^*_{\Pi(n\delta),2}}{\alpha(n)} 
\leq
\frac{\E {\Pi(n\delta)}}{\alpha(n)} 
\le \frac{n\delta}{n^{\theta/2}}=(n\delta)^{1-\theta/2}\delta^{\theta/2}\le x_0 \delta^{\theta/2}, \\
& \frac{\max(\E R_{\Pi(n t_2),2},\E R_{\Pi(n t_2),3})}{\alpha(n)} 
\le x_0 t_2^{\theta/2}.
\end{align*}
Now take $c=\max\{4, x_0 \}$. Since $t_2-t_1=\delta\ge 0$, then
for all $n\ge n_0$ we obtain
\begin{displaymath}
\frac{\var (M(nt_2)-M(nt_1))}{\alpha(n)}\le
2 c \frac{\delta^2}{t_2^{2-\theta/2}}+\delta^{\theta/2}+6 c \frac{t_1^2 \delta}{t_2^{3-\theta/2}}\le 9c \cdot \delta^{\theta/2}.
\end{displaymath}
 \end{proof}

We are ready to prove Theorem~\ref{Th1}. The proof is broken into four steps.

\paragraph{Step 1: Covariance.}
The first rather technical step consists in establishing a formulae for the
covariances which is put in Appendix.

\paragraph{Step 2: Convergence of finite-dimensional distributions.}
Along the lines of the proof of  \cite[Th.~12]{Dutko}, one can show
that for
\begin{displaymath}
  m\geq 1, \ \ \  0<t_1<t_2<\ldots<t_m\leq 1
\end{displaymath}
the triangular array of $m$-dimensional vectors (i.e.\ independent in $k$ for
every $n$)
\begin{displaymath}
\left\{\frac{\one (\Pi_k(nt_j)\ \text{is odd})-\P(\Pi_k(nt_j)\ \text{is odd})}{\sqrt{\beta(n)}}, \ j \leq m, \
k\leq n \right\}_{n\geq 1}
\end{displaymath}
satisfies the Lindeberg condition
(\seg~\cite[Th.~6.2]{BorovkovTV}). Similarly, the convergence of the
finite-dimensional distributions is shown for the process $M^*_n(t)$.

\paragraph{Step 3: Relative compactness.}

We shall follow the following plan:

\begin{itemize}
\item [({\bf a})] prove the continuity of the limiting process;
\item [({\bf b})] prove that $U_n^*$ and $U_n^{**}$ 
  ($M_n^*$ and $M_n^{**}$) are sufficiently close;
\item [({\bf c})] prove the relative compactness of $U_n^{**}$ ($M_n^{**}$).
\end{itemize}
\begin{itemize}

\item[{\bf a(U)}] 
 Take  $\tau_1=n t_1, \ \tau_2=n t_2$ for $0<t_1<t_2<0$. Then
\begin{multline*}
\E(U^*_n(t_2)-U^*_n(t_1))^2= \E \Big(  \sum\limits_{i=1}^{\infty}
(u_i-\ou_i)  \Big)^2/\beta(n) 
= \sum\limits_{i=1}^{\infty} \E (u_i-\ou_i)^2/\beta(n)\\
\le 5 \sum\limits_{i=1}^{\infty} \P (\Pi_i(\tau_2-\tau_1)>0) /\beta(n)= 5 \E R_{\Pi(\tau_2-\tau_1)}/\beta(n)
 \le 5 C(\theta) (t_2-t_1)^{\theta/2}.
\end{multline*}
We have used above the independence of the summands, inequality
(\ref{u}) and Lemma~\ref{Lem2.1}. 

Since the covariance function has a limit, \cite[Th.~1.4]{Adler} will
imply that the limiting Gaussian process a.s.\ has a continuous
modification on  $[0,1]$.

Since the trajectories of the limiting Gaussian process belong a.s.\ 
to the class $C(0,1)$, then the weak convergence in the Skorohod
topology implies the weak convergence in the uniform metric,
\seg~\cite{Billingsley}. Therefore, it is sufficient to prove the
relative compactness of $\{U^{*}_n\}_{n\ge n_0}$ (with $n_0$ as in
Lemma~\ref{Lem2.1}) in the Skorohod topology.

\item[{\bf b(U)}]
Since with probability one we have
\begin{displaymath}
  |U(nt)-U([nt])|\le \Pi(nt)-\Pi([nt])\le\Pi([nt]+1)-\Pi([nt]),
\end{displaymath}
then
\begin{displaymath}
  \E|U(nt)-U([nt])|\le1.
\end{displaymath}
Hence, for all $\eta>0$,
\begin{multline*}
  \P (\sup\limits_{0\le t\le 1} |U^*_n(t)-U^{**}_n(t)|>\eta)\\
  \le\P (\sup\limits_{0\le t \le 1} 
  (|U(nt)-U([nt])|+ \E|U(nt)-U([nt])|)>\eta\sqrt{\beta(n)})\\
  \le\P (\sup\limits_{0\le t \le 1} 
  (\Pi([nt]+1)-\Pi([nt])+ 1)>\eta\sqrt{\beta(n)})\\
  = \P (\sup\limits_{0\le m \le n} (\Pi(m+1)-\Pi(m)+1)>\eta\sqrt{\beta(n)})\\
  \le \sum\limits_{m=0}^n \P(\Pi(m+1)-\Pi(m)+1>\eta\sqrt{\beta(n)})\\
  \le \sum\limits_{m=0}^n \frac{\E e^{\Pi(m+1)-\Pi(m)+1}}{ e^{\eta\sqrt{\beta(n)}}}
  =(n+1)\frac{\E e^{\Pi(1)}}{e^{\eta\sqrt{\beta(n)}-1}}
  =(n+1)e^{e-\eta\sqrt{\beta(n)}}
  \to 0
\end{multline*}
when $ n\to\infty$. Therefore, it is sufficient to show the relative
compactness of $\{U^{**}_n\}_{n\ge n_0}$ (with $n_0$ as in
Lemma~\ref{Lem2.1}) in the Skorokhod topology.

\item[{\bf c(U)}]  For any $t_1, \ t_2 \in [0,1]$ satisfying $
  \frac{1}{2n}\le t_2-t_1$ we have that
\begin{displaymath}
[nt_2]-[nt_1]\le n(t_2-t_1)+1 \le n(t_2-t_1)+2n(t_2-t_1)= 3n(t_2-t_1)
\end{displaymath}
\begin{equation}\label{eq:star}
\le 3n(t_2-t_1)\cdot (2n(t_2-t_1))^3=24n^4(t_2-t_1)^4.
\end{equation}
Put $k=[16/\theta]+1$,  $\tau_1=[n t_1], \ \tau_2=[n t_2]$.
 
Recall the Rosenthal inequality~\cite{Rosenthal}: if $\phi_i$ are
independent random variables with $\E\phi_i=0$, then for all $k\geq2$
there exists a constant $c(k)$ such that
\begin{equation}
  \label{eq:rosen}
  \E\Big|\sum_i \phi_i\Big|^k\leq c(k)\max\bigg\{\sum_i\E|\phi_i|^k,
  \Big(\sum_i \E\phi_i^2\Big)^{k/2}\bigg\}.
\end{equation}

For all $n\ge n_0$ (with $n_0$ as in Lemma~\ref{Lem2.1}) we then have
\begin{multline*}
  \E|U^{**}_n(t_2)-U^{**}_n(t_1)|^k=\frac{\E\Big|\sum\limits_{i=1}^{\infty}
    (u_i-\ou_i)\Big|^k}{(\beta(n))^{k/2}} \\
  \leq \frac{c(k)}{(\beta(n))^{k/2}} \bigg( \sum\limits_{i=1}^{\infty}
  \E | u_i-\ou_i|^{k}+
  \Big( \sum\limits_{i=1}^{\infty} \E ( u_i-\ou_i)^2\Big)^{k/2}\bigg)\\
  \le \frac{C(k)}{(\beta(n))^{k/2}} \bigg(\sum\limits_{i=1}^{\infty}
  \P (\Pi_i(\tau_2-\tau_1)>0) + \Big(
  \sum\limits_{i=1}^{\infty} \P (\Pi_i(\tau_2-\tau_1)>0) \Big)^{k/2}\bigg)\\
  =\frac{C(k)}{(\beta(n))^{k/2}}
  \left(\E R(\tau_2-\tau_1) + \left(\E R(\tau_2-\tau_1) \right)^{k/2}\right)\\
  \le \frac{C(k)}{(\beta(n))^{k/2}}\left(24n^4(t_2-t_1)^4+ (\E
    R(3n(t_2-t_1)))^{k/2}\right) \le \widetilde{C}(\theta)(t_2-t_1)^4,
\end{multline*}
where $c(k)$, $C(k)$ and $\widetilde{C}(\theta)$ depend only on their arguments.

Above, we have used \eqref{eq:rosen} in the first inequality,
\eqref{u} in the second and finally, \eqref{eq:star} and Lemma~\ref{Lem2.1}
alongside with the bound
\begin{equation}
  \label{eq:2}
\E R(\tau_2-\tau_1) \leq \E (\Pi([nt_2])-\Pi([nt_1]))=[nt_2]-[nt_1].
\end{equation}

If $0\le t_2-t_1<\frac1n$, then $[nt_1]=[nt]$ or $[nt_2]=[nt]$ for all
$t\in[t_1, t_2]$, therefore
\begin{displaymath}
D\bydef  \E 
(|{U}^{**}_n(t)-{U}^{**}_n(t_1)|^{k/2}
|{U}^{**}_n(t_2)-{U}^{**}_n(t)|^{k/2})=0\le (t_2-t_1)^2.
\end{displaymath}
If $t_2-t_1\ge 1/n$, then there are the following three cases:
\begin{enumerate}
\item if $t_2-t\ge \frac1{2n}$, $t-t_1\ge \frac1{2n}$, then the
  Cauchy--Schwarz inequality implies
  \begin{displaymath}
    D\le \widetilde{C}(\theta) (t_2-t)^2\cdot (t-t_1)^2 \le
    \widetilde{C}(\theta)(t_2-t_1)^2.
  \end{displaymath}
\item If $t_2-t\ge \frac1{2n}$, $t-t_1< \frac1{2n}$, then since
  \begin{displaymath}
|U([nt])-U([nt_1])|\le_{\textrm{a.s.}} \Pi([nt])-\Pi([nt_1])\le_{st}\Pi(1),
\end{displaymath}
the same inequality yields
\begin{displaymath}
D\le \left( \widetilde{C}(\theta) (t_2-t)^4 \cdot\E
  \left(\frac{\Pi(1)+1}{\sqrt{\beta(n)}}\right)^{k}\right)^{1/2}\le
\widehat{C}(\theta)(t_2-t_1)^2.
\end{displaymath}
\item If $t_2-t< \frac1{2n}$, $t-t_1\ge \frac1{2n}$, then since
  \begin{displaymath}
  |U([nt_2])-U([nt])|\le_{\textrm{a.s.}} \Pi([nt_2])-\Pi([nt])\le_{st}\Pi(1),
\end{displaymath}
we have that
\begin{displaymath}
D\le \left( \E \left(\frac{\Pi(1)+1}{\sqrt{\beta(n)}}\right)^{k}\cdot
  \widetilde{C}(\theta) (t-t_1)^4\right)^{1/2}\le
\widehat{C}(\theta)(t_2-t_1)^2. 
\end{displaymath}
Now the relative compactness follows from,
e.g.,~\cite[Th.~13.5]{Billingsley}.
\end{enumerate}

\item[{\bf a(M)}] Because the covariance function has a limit, it is
  sufficient to appeal to Lemma~\ref{Lem2.3} and
  \cite[Th.~1.4]{Adler} to establish existence of an almost sure
  continuous on $[0,1]$ modification of the limiting Gaussian
  process. Since the trajectories of this process are a.s. in
  $C(0,1)$, then the weak convergence in the Skorohod topology implies
  the uniform convergence, see \cite{Billingsley}. Thus it is
  sufficient to prove a relative compactness of the family
  $\{M^{*}_n\}_{n\ge n_0}$ in the Skorohod topology (here $n_0$ is the
  same as in Lemma~\ref{Lem2.1}).

\item[{\bf b(M)}]
Set $\tau_2=nt$ and $\tau_1=[nt]$. Since $\tau_2-\tau_1\le 1$, then
\begin{displaymath}
\E |M(\tau_2)-M(\tau_1)|\le\sum_{i=1}^{\infty}(\tau_2-\tau_1)p_i
e^{-p_i \tau_2} + \tau_1 p_i e^{-p_i \tau_1}(1-e^{-p_i
  (\tau_2-\tau_1)}) 
\end{displaymath}
\begin{displaymath}
\le\sum_{i=1}^{\infty} p_i e^{-p_i \tau_2} +  e^{-1}p_i
(\tau_2-\tau_1)<\sum_{i=1}^{\infty} 2p_i=2. 
\end{displaymath}
Let $m'''_i = m''_i(\tau_1,\tau_1+1)$ and $\om'''_i=\E m'''_i$. Then
we have almost surely,
\begin{multline*}
 |M(\tau_2)-M(\tau_1)|\le\sum_{i=1}^{\infty} (m'_i+m''_i) \le
 \sum_{i=1}^{\infty} (p_i +m'''_i)\\
 =1+\sum_{i=1}^{\infty} (m'''_i + \om'''_i - \om'''_i)
<2+\left| \sum_{i=1}^{\infty}  (m'''_i - \om'''_i) \right|.
\end{multline*}
We know that for any integer $k\ge2$
\begin{displaymath}
\E | m'''_i - \E m'''_i |^k<2^k k! \P (\Pi_i(\tau_1+1-\tau_1)>0)=2^k
k!(1-e^{-p_i})< 2^k k! p_i. 
\end{displaymath}
Using the independence of the terms and Rosenthal inequality, for any
$k\geq 2$,
\begin{displaymath}
\E \left| \sum_{i=1}^{\infty}  (m'''_i - \om'''_i )\right|^k\le c(k)
\left(
\sum\limits_{i=1}^{\infty} \E | m'''_i - \om'''_i |^k+
\left( \sum\limits_{i=1}^{\infty} \E ( m'''_i - \om'''_i )^2\right)^{k/2}\right)
\end{displaymath}
\begin{displaymath}
<c(k)(2^kk! + 4^k)=C(k).
\end{displaymath}
Hence, for $k\ge [2/\theta]+1$ and all $\eta>0$
\begin{displaymath}
\P (\sup\limits_{0\le t\le 1} |M^*_n(t)-M^{**}_n(t)|>\eta)
\end{displaymath}
\begin{displaymath}
\le\P (\sup\limits_{0\le t \le 1} 
(|M(nt)-M([nt])|+ \E|M(nt)-M([nt])|)>\eta\sqrt{\alpha(n)})
\end{displaymath}
\begin{displaymath}
\le\P \left(\max\limits_{0\le [nt] \le n} 
\left(\left| \sum_{i=1}^{\infty}  m'''_i - \E m'''_i \right|+ 4\right)>\eta\sqrt{\alpha(n)}\right)
\end{displaymath}
\begin{displaymath}
\le \sum\limits_{[nt]=m\in\{0,1,...,n\}} \P \left( \left| \sum_{i=1}^{\infty}  m'''_i - \E m'''_i \right|+ 4>\eta\sqrt{\alpha(n)}\right)
\end{displaymath}
\begin{displaymath}
\le \sum\limits_{m=0}^n \frac{C(k)}{ (\eta\sqrt{\alpha(n)}-4)^k}
=\frac{C(k)(n+1)}{ (\eta\sqrt{\alpha(n)}-4)^k}
\to 0 \ \textrm{when}  \  n\to\infty. 
\end{displaymath}
Therefore, it is sufficient to show the local compactness of
$\{M^{**}_n\}_{n\ge n_0}$ in the Skorohod topology.

\item[{\bf c(M)}] Let $t_1, \ t_2 \in [0,1]$ and
  $ \frac{1}{2n}\le t_2-t_1$, then \eqref{eq:star} holds. Set
  $k=[16/\theta]+1$, $\tau_1=[n t_1], \ \tau_2=[n t_2]$.
 
Again, by independence and the Rosenthal inequality,
\begin{multline*}
  \E|M^{**}_n(t_2)-M^{**}_n(t_1)|^{k}=\frac{\E\left|
      \sum\limits_{i=1}^{\infty}(m_i-\om_i)\right|^k}{(\alpha(n))^{k/2}}
\\
\leq \frac{c(k)}{(\alpha(n))^{k/2}}
\left(
\sum\limits_{i=1}^{\infty} \E | m_i-\om_i|^{k}+
\left( \sum\limits_{i=1}^{\infty} \E ( m_i-\om_i)^2\right)^{k/2}\right)
\\
\le
 \frac{C(\beta)}{(\alpha(n))^{k/2}}
\left(\sum\limits_{i=1}^{\infty} \P (\Pi_i(\tau_2-\tau_1)>0) + \left(
\var  (M(\tau_2)-M(\tau_1)) \right)^{k/2}\right)
\\
=\frac{C(k)}{(\alpha(n))^{k/2}}
\left(\E R(\tau_2-\tau_1) + \left(\var  (M(\tau_2)-M(\tau_1)) \right)^{k/2}\right)
\\
\le
\frac{C(k)}{(\alpha(n))^{k/2}}\left(24n^4(t_2-t_1)^4+
(C(\theta) \alpha(n)(\tau_2-\tau_1)/n)^{k/2}\right)
\le
\widetilde{C}(\theta)(t_2-t_1)^4,
\end{multline*}
where $c(k)$, $C(k)$ and $\widetilde{C}(\theta)$ depend only on their
arguments.

Above, we have used inequalities (\ref{m}), (\ref{eq:star}) and
Lemmas~\ref{Lem2.3}, \ref{Lem2.1} alongside with the bound
\begin{displaymath}
\E R(\tau_2-\tau_1) \leq \E (\Pi([nt_2]-[nt_1]))=[nt_2]-[nt_1].
\end{displaymath}

When $0\le t_2-t_1<\frac1n$, then $[nt_1]=[nt]$ or $[nt_2]=[nt]$ for any $t\in[t_1, t_2]$.  
Thus
\begin{displaymath}
B\bydef  \E 
(|{M}^{**}_n(t)-{M}^{**}_n(t_1)|^{k/2}
|{M}^{**}_n(t_2)-{M}^{**}_n(t)|^{k/2})=0\le (t_2-t_1)^2.
\end{displaymath}

When $t_2-t_1\ge 1/n$, we have the following three cases:
\begin{enumerate}
\item if $t_2-t\ge \frac1{2n}$, $t-t_1\ge \frac1{2n}$, then the
  Cauchy--Schwarz inequality gives
  \begin{displaymath}
B\le \widetilde{C}(\theta) (t_2-t)^2\cdot (t-t_1)^2 \le \widetilde{C}(\theta)(t_2-t_1)^2;
\end{displaymath}
\item if $t_2-t\ge \frac1{2n}$, $t-t_1< \frac1{2n}$, then since for any $l\ge 2$,
\begin{displaymath}
 \E |M([nt])-M([nt_1])-\E(M([nt])-M([nt_1])|^l
\end{displaymath}
\begin{displaymath}
\le\E\left(4+\left| \sum_{i=1}^{\infty}  m''_i([nt_1]+1,[nt_1]) - \E m''_i([nt_1]+1,[nt_1]) \right|\right)^l< C(l),
\end{displaymath}
the Cauchy--Schwarz inequality yields the bound
\begin{displaymath}
  B\le \left( \widetilde{C}(\theta) (t_2-t)^4 \cdot
    \frac{C(k)}{\alpha(n)^{k/2}}\right)^{1/2}\le \widehat{C}(\theta)(t_2-t_1)^2;
\end{displaymath}
\item finally, $t_2-t< \frac1{2n}$, $t-t_1\ge \frac1{2n}$, is similar
  to the previous case.
\end{enumerate}
Thus the required compactness follows from~\cite[Th.~13.5]{Billingsley}.
\end{itemize}

Finally, for the next step we need to show that $M(s)$, when time scaled, is close to
its fully Poissonised version
\begin{displaymath}
  \widetilde{M}(s)\bydef M_{\Pi(s)}=\sum_{i=1}^\infty \Pi(s) p_i\one_{\Pi_i(s)=0}.
\end{displaymath}
Namely, we aim to show that
\begin{equation}\label{eq:mt}
  \sup\limits_{0\le t\le 1} |M^*_n(t)-\widetilde{M}_n(t)|\to 0\quad \text{in probability,}
\end{equation}
where
\begin{displaymath}
  \widetilde{M}_n(t)=\frac{\widetilde{M}(nt)-\E\widetilde{M}(nt)}{(\alpha(n))^{1/2}}.
\end{displaymath}
Introduce $\Pi'_i(s)=\Pi(s)-\Pi_i(s)$ and
$\widetilde{\Pi}(s)=(\Pi(s)-s)/\sqrt{s}$. Since
$\widetilde{M}(s)=\sum_{i=1}^\infty \Pi'_i(s) p_i \one_{\Pi_i(s)=0}$, then
\begin{multline*}
  |\E\widetilde{M}(s)-\E M(s)|=|\E\sum_{i=1}^\infty
  (\Pi'_i(s)-s) p_i \one_{\Pi_i(s)=0}| \\
  = \Big|\sum_{i=1}^\infty (s(1-p_i)-s) p_i e^{-s p_i}\Big|=\frac{2 \E
    R_{\Pi(s),2}}{s}\to 0
\end{multline*}
as $s\ti$ and it is bounded by 1. Thus there exists a sufficiently small
$\epsilon=\epsilon(\theta)>0$ such that for $\delta_n=n^{\epsilon-1}$
\begin{displaymath}
\sup\limits_{0\le t\le \delta_n} |M^*_n(t)-\widetilde{M}_n(t)|<
\frac{\Pi(n\delta_n)+n\delta_n+1}{(\alpha(n))^{1/2}} \to 0\ \text{a.s.}
\end{displaymath}
when $n\to\infty$.

By the Strong Law of Large Numbers for $M(s)$ and the well-known
asymptotic behaviour of $\E M(s)$ (see,
e.g.,~\cite[Eq.~(23)]{Karlin}), we conclude that for any
$\theta\in(0,1]$, $M(s)/(s \alpha(s))^{1/2}\to 0$ a.s. when
$s\to\infty$. Moreover, according to the Central Limit theorem
$\widetilde{\Pi}(s)$ is asymptotically standard normal for large $s$.

Finally, we have almost surely,
\begin{displaymath}
|M^*_n(t)-\widetilde{M}_n(t)|\le \frac{|\widetilde{\Pi}(nt)|
  M(nt)}{(nt \alpha(n))^{1/2}}+\frac{1}{(\alpha(n))^{1/2}}.
\end{displaymath}
Using this inequality, the fact that $\sup\limits_{0\le t\le
  1}(\cdot)\le\sup\limits_{0\le t\le
  \delta_n}(\cdot)+\sup\limits_{\delta_n\le t\le 1}(\cdot)$ and that
$\sup\limits_{0\le t\le 1}(\cdot)$ is a continuous functional, we
readily obtain~\ref{eq:mt}.

\paragraph{Step 4: Approximation of the initial process.}
Since $\Pi(t)$ is monotone, the Strong Law of Large Numbers implies
that for any $\epsilon, \delta\in(0,1)$ there is an integer
$N=N(\epsilon,\delta)$ such that for all $n\ge N$ one has
\begin{displaymath}
\P(\forall t\in[0,1] \  \ \exists \tau: |\tau-t|\le \delta,  \ \Pi(n\tau)= [nt]) \bydef  \P (A(n))\ge1-\epsilon,
\end{displaymath}
see Lemma~\ref{Lem2.2}. Here and below, $F$ stands for $R, U$ or
$M$. The relative compactness of the distributions $\{F_n^*\}_{n\geq
  n_0}$ implies that for any $\epsilon\in(0,1)$ and $\eta>0$ there
exist $\delta\in(0,1)$ and an integer $N_1=N_1(\epsilon,\eta)$ such
that for all $n\ge N_1$,
\begin{displaymath}
 \P(\sup\limits_{|t-\tau| \leq \delta}
 \left|F^*_n(\tau)-F^*_n(t)\right|\ge \eta) \le \epsilon. 
\end{displaymath}
Hence, since
\begin{displaymath}
\P(F_n(t)=F^*_n(\tau)|\Pi(n\tau)=[nt])=1,
\end{displaymath}
then for all $n\ge \max(N,N_1)$,
\begin{multline*}
\P\left(\sup\limits_{0 \leq t \leq 1} \left|F_n(t)-F^*_n(t)\right|\ge \eta\right)\le 
\P\left(\sup\limits_{0 \leq t \leq 1} \left|F_n(t)-F^*_n(t)\right|\ge
  \eta, A(n)\right)+\epsilon \\
\le \P\left(\sup\limits_{|t-\tau| \leq \delta} \left|F^*_n(\tau)-F^*_n(t)\right|\ge \eta\right)+\epsilon\le 2 \epsilon.
\end{multline*}
which proves Theorem~\ref{Th1}.

\paragraph{Acknowledgements.} MC's research is supported
by RSF Grant 17-11-01173-Ext. He also
acknowledges hospitality of Chalmers university where a part of this
work has been done. The authors are thankful to Sergey Foss
for his interest in this research and valuable comments and to the
anonymous reviewer for thorough reading and spotting some inaccuracies in the previous
version of the manuscript.

\section*{Appendix}
An explicit expression for the covariance between $R(\tau)$ and $R(t)$ can
be found in~\cite{Chebunin2}. Take $\tau\leq t$. The
\begin{multline*}
  c^*_{UU} (\tau,t)=\cov (U(\tau), U(t)) \\
  =\sum_{k=1}^{\infty} \P (\Pi_k(\tau), \Pi_k(t)\ \text{is odd} )-
  \P (\Pi_k(\tau)\ \text{is odd}) \P (\Pi_k(t)\ \text{is odd}) 
  \\
  = \frac14
  \sum_{k=1}^{\infty} \bigg(
  (1-e^{-2p_k\tau})(1+e^{-2p_k(t-\tau)})-(1-e^{-2p_k\tau})(1-e^{-2p_kt}) \bigg)
  \\
  =\frac14 \sum_{k=1}^{\infty} e^{-2p_k(t-\tau)}-e^{-2p_k(t+\tau)}
  =\frac12\E(U(t+\tau)-U(t-\tau)).
\end{multline*}
Hence (since $\frac{\beta(nt)}{\beta(n)} \to t^{\theta}$ as $n \to \infty$)
\begin{align*}
  c_{\upsilon\upsilon}(\tau,t)&=\lim_{n\to \infty} \frac{c_{UU}^*(n\tau,
    nt)}{\alpha(n)}
  =\Gamma(1-\theta)2^{\theta-2}((t+\tau)^\theta-(t-\tau)^\theta), \theta\in(0,1),
\\
c_{\upsilon\upsilon}(\tau,t)&=\lim_{n\to \infty} \frac{c_{UU}^*(n\tau, nt)}{n
  L^*(n)} =2\tau, \  \theta=1.
\end{align*}
cf.~\cite[Eq.~(21)]{Karlin}.

Next, 
\begin{multline*}
  c_{MM}^{*} (\tau,t)=\cov (M(\tau), M(t)) 
  \\
  =\sum_{k=1}^{\infty} \E (tp_i \one (\Pi_i(t)=0)- tp_i e^{-tp_i})
  (\tau p_i \one (\Pi_i(\tau)=0)- \tau p_i e^{-\tau p_i}) 
  \\
  = \sum_{k=1}^{\infty} t\tau p_i^2 e^{-tp_i}(1-e^{-\tau p_i})
  =\frac{2\tau}{t}\E R_{\Pi(t),2}-\frac{2t \tau}{(t+\tau)^2}\E R_{\Pi(t+\tau),2}.
\end{multline*}
Since $\frac{\alpha(nt)}{\alpha(n)} \to t^{\theta}$ when $n \to \infty$),
\begin{displaymath}
  c_{\mu\mu}(\tau,t)=\lim_{n\to \infty} \frac{c_{MM}^{*}(n\tau,
    nt)}{\alpha(n)}
  =\theta \Gamma(2-\theta)\left( \frac{\tau}{t^{1-\theta}}
    - \frac{t\tau}{(t+\tau)^{2-\theta}} \right),
\end{displaymath}
cf.~\cite[Eq.~(23)]{Karlin}.

Continuing,
\begin{multline*}
c^*_{RU} (\tau,t)=\cov (R(\tau), U(t)) =
\sum_{k=1}^{\infty} \cov(1-\one (\Pi_k(\tau)=0), \one (\Pi_k(t)\ \text{is odd}))
\\
 =
-\sum_{k=1}^{\infty} \cov(\one (\Pi_k(\tau)=0), \one (\Pi_k(t)\ \text{is odd}))
\\
= -
\sum_{k=1}^{\infty} \P (\Pi_k(\tau)=0, \Pi_k(t)\ \text{is odd} )-
\P (\Pi_k(\tau)=0) \P (\Pi_k(t)\ \text{is odd}) 
\\
= -\frac12
\sum_{k=1}^{\infty} \bigg( e^{-p_k\tau}(1-e^{-2p_k(t-\tau)})- e^{-p_k\tau}(1-e^{-2p_kt}) \bigg)
\\
= \frac12
\sum_{k=1}^{\infty} \bigg( e^{-p_k(2t-\tau)} - e^{-p_k(2t+\tau)}\pm 1\bigg)
=\frac12\E(R(2t+\tau)-R(2t-\tau)).
\end{multline*}
Similarly,
\begin{multline*}
c^*_{RU}(t,\tau)=\cov (R(t), U(\tau)) =
-\sum_{k=1}^{\infty} \cov(\one (\Pi_k(t)=0), \one (\Pi_k(\tau)\ \text{is odd}))
\\
= \frac12
\sum_{k=1}^{\infty}  e^{-p_kt}(1-e^{-2p_k\tau})
= \frac12
\sum_{k=1}^{\infty} \bigg( e^{-p_k t} - e^{-p_k(2\tau+t)}\pm 1\bigg)
\\
=\frac12\E(R(2t+\tau)-R(t)).
\end{multline*}

Because $\frac{\beta(nt)}{\beta(n)} \to t^{\theta}$ when $n \to
\infty$, for $\theta\in(0,1)$ we have that
\begin{align*}
c_{\rho\upsilon}(\tau,t) & =\lim_{n\to \infty} \frac{c_{RU}^*(n\tau,
                 nt)}{\alpha(n)}
                 =\Gamma(1-\theta)((2t+\tau)^{\theta} - (2t-\tau)^{\theta})/2,
\\
c_{\rho\upsilon}(t,\tau) &=\lim_{n\to \infty} \frac{c^*_{RU}(nt, n\tau)}{\alpha(n)}=\Gamma(1-\theta)((2t+\tau)^\theta - t^\theta)/2.
\end{align*}
For $\theta=1$ this reduces to
\begin{align*}
c_{\rho\upsilon}(\tau,t)&=\lim_{n\to \infty} \frac{c_{RU}^*(n\tau, nt)}{n L^*(n)}=\tau,
\\
c_{\rho\upsilon}(t,\tau)&=\lim_{n\to \infty} \frac{c^*_{RU}(nt, n\tau)}{nL^*(n)}=(t+\tau)/2.
\end{align*}
cf.~\cite[Th.~1]{Karlin}.

Next,
\begin{multline*}
c^*_{MU} (\tau,t)=\cov (M(\tau), U(t)) 
 =
\sum_{k=1}^{\infty} \tau p_k \cov(\one (\Pi_k(\tau)=0), \one (\Pi_k(t)\ \text{is odd}))
\\
= \frac12
\sum_{k=1}^{\infty}\tau p_k \bigg( e^{-p_k(2t+\tau)}-e^{-p_k(2t-\tau)}\bigg)
\\
=\frac{\tau}{2(2t+\tau)}\E M(2t+\tau)-\frac{\tau}{2(2t-\tau)}\E M(2t-\tau).
\end{multline*}
and
\begin{multline*}
c^*_{MU} (t,\tau)=\cov (M(t), U(\tau)) = \frac12
\sum_{k=1}^{\infty}t p_k \bigg(e^{-p_k(2\tau+t)}-e^{-p_k t}\bigg)
\\
=\frac{t}{2(2\tau+t)}\E M(2\tau+t)-\frac12\E M(t).
\end{multline*}

Finally,
\begin{multline*}
c^*_{RM} (\tau,t)=\cov (R(\tau), M(t)) =
\sum_{k=1}^{\infty} \cov(1-\one (\Pi_k(\tau)=0), t p_k\one (\Pi_k(t)=0))
\\
 =
-\sum_{k=1}^{\infty} t p_k \cov(\one \{\Pi_k(\tau)=0\}, \one \{\Pi_k(t)=0)\}
\\
= -\sum_{k=1}^{\infty}t p_k \bigg( e^{-p_k t}- e^{-p_k(\tau+t)} \bigg)
=\frac{t}{\tau+t}\E M(\tau+t)-\E M(t).
\end{multline*}
and
\begin{displaymath}
c^*_{RM} (t,\tau) =\cov (R(t), M(\tau))
=\frac{\tau}{\tau+t}\E M(\tau+t)-\frac{\tau}{t}\E M(t).
\end{displaymath}

Because $\frac{\alpha(nt)}{\alpha(n)} \to t^{\theta}$ when $n \to
\infty$, for $\theta\in(0,1)$ we obtain
\begin{align*}
  c_{\rho\mu}(\tau,t) & =\lim_{n\to \infty}
                   \frac{c^{*}_{RM} (n\tau, nt)}{\alpha(n)}=\theta
                   \Gamma(1-\theta)
                   \left( \frac{t}{(t+\tau)^{1-\theta}} - t^\theta \right),
\\
c_{\rho\mu}(t,\tau)&=\lim_{n\to \infty} \frac{c^{*}_{RM} (nt,
                n\tau)}{\alpha(n)}
                =\theta \Gamma(1-\theta)\left(
                \frac{\tau}{(t+\tau)^{1-\theta}}
                -\frac{\tau}{ t^{1-\theta}} \right),
\\
c_{\mu\upsilon}(\tau,t)&=\lim_{n\to \infty} \frac{c^{*}_{MU} (n\tau,
                nt)}{\alpha(n)}
                =\theta \Gamma(1-\theta)\left(
                \frac{\tau}{2(2t+\tau)^{1-\theta}}
                -\frac{\tau}{2(2t-\tau)^{1-\theta}} \right),
\\
c_{\mu\upsilon}(t,\tau)&=\lim_{n\to \infty} \frac{c^{*}_{MU} (nt,
                n\tau)}{\alpha(n)}
                =\theta \Gamma(1-\theta)\left(
                \frac{t}{2(2\tau+t)^{1-\theta}}
                -\frac{t^\theta}{2} \right),
\end{align*}
cf. \cite[Eq.~(23)]{Karlin}.

Clearly, $L(n)\to 0$ as $n\to\infty$. According to
\cite[Lem.~4]{Karlin}, in the case $\theta=1$ the function $L^*(n)\to
0$ when $n\to\infty$ is slowly varying and
\begin{equation}
\label{L^*}
\lim_{n\to\infty} \frac{L(n)}{L^*(n)}\bydef \lim_{n\to\infty}\delta_n=0.
\end{equation}
Therefore, in the case $\theta=1$,
\begin{align*}
c_{\rho\mu}(\tau,t)&=\lim_{n\to \infty} \frac{c^{*}_{RM} (n\tau,
                nt)}{\alpha(n)} \sqrt{\delta_n}=0, \quad
c_{\rho\mu}(t,\tau)&=\lim_{n\to \infty} \frac{c^{*}_{RM} (nt,
                n\tau)}{\alpha(n)}\sqrt{\delta_n}=0,\\
c_{\mu\upsilon}(\tau,t)&=\lim_{n\to \infty} \frac{c^{*}_{MU} (n\tau,
                nt)}{\alpha(n)}
                \sqrt{\delta_n}=0, \quad
c_{\mu\upsilon}(t,\tau)&=\lim_{n\to \infty} \frac{c^{*}_{MU} (nt,
                n\tau)}{\alpha(n)}
                \sqrt{\delta_n}=0.
\end{align*}

 % \bibliographystyle{plain}
 % \bibliography{urns}

\end{document}